\documentclass{amsart}
\usepackage{mathptmx, amssymb, enumerate, colonequals, amscd, bm, amsmath,mathtools}
\usepackage{color}
\usepackage[colorlinks=true, pagebackref, hyperindex, citecolor=green, linkcolor=red]{hyperref}
\usepackage{subdepth}

\DeclareMathOperator{\characteristic}{char}

\newtheorem{theorem}{Theorem}[section]

\newtheorem{proposition}[theorem]{Proposition}
\newtheorem{corollary}[theorem]{Corollary}

\theoremstyle{definition}

\theoremstyle{remark}

\numberwithin{equation}{section}

\begin{document}

\title{The $F$-regularity of algebraic sets related to commutator matrices}

\author{Trung Chau}
\address{Department of Mathematics, University of Utah, 155 South 1400 East, Salt Lake City, UT~84112, USA}
\email{trung.chau@utah.edu}

\begin{abstract}
We study the algebraic sets of pairs of matrices defined by the vanishing of the anti-diagonal as well as the cross-diagonal of their commutator matrix. We prove that, over a field of prime characterisitic, the coordinate ring of the latter is always $F$-regular and, with exactly one exception, so is that of the former, thus proving a conjecture of Kadyrsizova and Yerlanov.
\end{abstract}
\maketitle

\section{Introduction}

Let $X=(x_{ij})_{1\leq i,j \leq n}$ and $Y=(y_{ij})_{1\leq i,j \leq n}$ be $n\times n$ matrices of indeterminates over a field ~$\Bbbk$. The \textit{commutator matrix} of $X$ and $Y$ is $C\coloneqq XY-YX=(c_{ij})_{1\leq i,j \leq n}$. Let $\mathfrak{d}$ and $\mathfrak{a}$ denote the ideals generated by the diagonal and the anti-diagonal entries of $C$, respectively. Then $\mathfrak{c}\coloneqq \mathfrak{d} +  \mathfrak{a}$ is the ideal generated by the \textit{cross-diagonal} of $C$.

% The algebraic sets defined by commutator matrices have been studied extensively since at least the 1950s. Such algebraic sets were proved to be irreducible by Motzkin and Taussky \cite{MT55}, Gerstenhaber \cite{Ge61} and Guralnick \cite{Gu92} in different generalizations.  Thompson \cite{Th85} in 1985 showed that this variety is reduced and Cohen-Macaulay when $n=3$. It is still an open problem whether this result extends to $n\geq 4$.

% Since then, algebraic sets related to commutator matrices have been considered and studied extensively. Knutson \cite{Kn05} introduced the algebraic set of nearly commuting $n\times n$ matrices, i.e., the set defined by the vanishing of all off-diagonal entries of $C$. He conjectured that this algebraic set is a reduced complete intersection that has two irreducible components, one of which is the variety of commuting $n\times n$ matrices. This result was proved by Young \cite{Yo21}; in the same paper, Young also studied the algebraic set of zero diagonal commutator, i.e., the set defined by the vanishing of  $\mathfrak{d}(C)$, and showed that it is an irreducible complete intersection. 

Algebraic sets defined by commutator matrices have been studied extensively since at least the 1950s. Such algebraic sets were proved to be irreducible by Motzkin and Taussky \cite{MT55}, Gerstenhaber \cite{Ge61}, and Guralnick \cite{Gu92} under varying assumptions. Hochster (1984) conjectured that these algebraic sets are reduced. It is straightforward to verify this when $n\leq 2$. Soon after, Thompson \cite{Th85} settled the conjecture in the case $n=3$; the conjecture remains open for $n\geq 4$.  Knutson \cite{Kn05} introduced the algebraic set of \textit{nearly commuting matrices}, i.e., the set defined by the vanishing of all off-diagonal entries of $C$, and conjectured that it is a reduced complete intersection, with two irreducible components, one of which is the variety of commuting $n\times n$ matrices. This was proved by Young \cite{Yo21} who also studied the algebraic set of matrices with zero diagonal commutator, i.e., the set defined by the vanishing of  $\mathfrak{d}$, and showed that it is an irreducible complete intersection. 

Since the introduction of tight closure \cite{HH90}, it has been a question of considerable interest to determine the $F$-singularities, i.e., the tight closure properties, of various algebraic sets. The classic examples that have been studied extensively are determinantal, symmetric determinantal, and Pfaffian rings, which are all known to be $F$-regular, see for example \cite[Theorem 7.14]{HH94b}. How about coordinate rings defined by the aforementioned algebraic sets related to commutator matrices? Partial answers were recently obtained by Kadyrsizova \cite{Ka18} and Kadyrsizova and Yerlanov \cite{KY22}:

\begin{theorem}\label{Fsingularities}
    Let $\Bbbk$ be a field of positive characteristic, $X$ and $Y$  be $n\times n$ matrices of indeterminates, $C=(c_{ij})_{1\leq i,j \leq n}$ their commutator matrix,   and $\mathfrak{d}$, $\mathfrak{a}$, and $\mathfrak{c}$ are as defined above. Then:
    \begin{enumerate}
        \item \cite[Theorem 6]{Ka18} The ring $\Bbbk[X,Y]/(c_{ij})_{1\leq i,j\leq n, ~i\neq j}$ is $F$-pure in the cases $n\leq 3$.
        \item \cite[page 30]{KY22} The ring $\Bbbk[X,Y]/(c_{ij})_{1\leq i,j\leq n, ~i\neq j}$ is not $F$-pure when $n=4$ and $\characteristic(\Bbbk)=2$.
        \item \cite[Theorem 2.3]{KY22} The ring $\Bbbk[X,Y]/\mathfrak{d}$ is $F$-regular.
        \item \cite[Theorems 3.7, 3.9, 4.3]{KY22} The rings $\Bbbk[X,Y]/\mathfrak{c}$ and $\Bbbk[X,Y]/\mathfrak{a}$ are $F$-pure.
    \end{enumerate}
\end{theorem}

In this paper, we  determine precisely when the rings $\Bbbk[X,Y]/\mathfrak{c}$ and $\Bbbk[X,Y]/\mathfrak{a}$ are $F$-regular, thus proving \cite[Conjectures 5.1, 5.2]{KY22}. Our main result is:

\begin{theorem}\label{thm:main}
    Let $\Bbbk$ be a field of positive characteristic, $X$ and $Y$ be $n\times n$ matrices of indeterminates, and $\mathfrak{c}$ and $\mathfrak{a}$ denote the ideals generated by cross-diagonal and anti-diagonal entries of the commutator matrix $XY-YX$, respectively.  Then  $\Bbbk[X,Y]/\mathfrak{c}$ is $F$-regular for each $n$, while $\Bbbk[X,Y]/\mathfrak{a}$ is $F$-regular  exactly when $n\neq 2$.
\end{theorem}

Using \cite[Theorem 4.3]{Sm97}, one immediately obtains the characteristic-zero version:

\begin{corollary}
    Let $\Bbbk$ be a field of characteristic zero, $X$ and $Y$ be $n\times n$ matrices of indeterminates, and $\mathfrak{c}$ and $\mathfrak{a}$ denote the ideals generated by cross-diagonal and anti-diagonal entries of the commutator matrix $XY-YX$, respectively.  Then $\Bbbk[X,Y]/\mathfrak{c}$ has rational singularities for each $n$, while $\Bbbk[X,Y]/\mathfrak{a}$ has rational singularities  exactly when $n\neq 2$.
\end{corollary}

\begin{corollary}
    In any characteristic, the ideal $\mathfrak{c}$ is prime for each $n$, while the ideal $\mathfrak{a}$ is prime exactly when $n\neq 2$.
\end{corollary}

\section{Preliminaries}

We recall the relevant definitions. Let $R$ be a Noetherian commutative ring of prime characteristic $p>0$ and $R^\circ$ denote the complement of all minimal primes of $R$. The ring $R$ is \textit{strongly $F$-regular} if for each $c\in R^\circ$, there exists $q=p^e$ such that the $R$-linear map $R\to R^{1/q},\ 1\mapsto c^{1/q}$ splits. There are different notions of $F$-regularity, namely weak $F$-regularity, $F$-regularity, and strong $F$-regularity. For an $\mathbb{N}$-graded ring, these  coincide, so we do not make a distinction, and will only use \textit{$F$-regular}; the equivalence was proved by Lyubeznik and Smith \cite[Corollaries 4.3, 4.4]{LS99} if the ring is \textit{$F$-finite}, i.e., the embedding $R\to R^{1/p}$ is module-finite, and by Hochster \cite[page 237]{Ho07} in general. From now on, unless otherwise stated, we will assume that $R$ is $\mathbb{N}$-graded.

A weaker condition than $F$-regularity is $F$-purity. A ring $R$ is \textit{$F$-pure} if the embedding $R\to R^{1/p}$ is pure, i.e., the induced map
    \[ R\otimes M \to R^{1/p}\otimes M \]
    is injective for any $R$-module $M$.
    
Although these two properties do not deform in general (\cite[Theorem 1.1]{Si99} and \cite[Example ~4.8]{Fe83}), they do deform for Gorenstein rings:

\begin{theorem}\label{thm:fregdeform}(\cite[Theorem 4.2]{HH94a}, \cite[Theorem 3.4]{Fe83})
    Let $f$ be an $R$-regular element. If $R$ is Gorenstein and $R/fR$ is $F$-regular (respectively, $F$-pure), then $R$ is $F$-regular (respectively, $F$-pure).
\end{theorem}

Useful ways to check $F$-regularity and $F$-purity are Glassbrenner's and Fedder's criteria, recorded below in the graded case. For an ideal $I$ and a prime power $q=p^e$, let $I^{[q]}$ denote the ideal generated by the $q^{\text{th}}$ powers of the elements in $I$.

\begin{theorem}(Glassbrenner's criterion, \cite[Theorem 3.1]{Gl96})
    Let $S$ be a polynomial ring over an $F$-finite field of prime characteristic $p>0$ and $\mathfrak{m}$ denote its homogeneous maximal ideal. If $I$ is a homogeneous ideal, then $S/I$ is $F$-regular if and only if for each $c\in (S/I)^\circ$, we have $c(I^{[q]}:_S I)\nsubseteq \mathfrak{m}^{[q]}$ for some $q=p^e$.
\end{theorem}

\begin{theorem}(Fedder's criterion, \cite[Theorem 1.12]{Fe83})
    Let $S$ be a polynomial ring over an $F$-finite field of prime characteristic $p>0$ and $\mathfrak{m}$ denote its homogeneous maximal ideal. If $I$ is a homogeneous ideal, then $S/I$ is $F$-pure if and only if $(I^{[p]}:_S I)\nsubseteq \mathfrak{m}^{[p]}$.
\end{theorem}

In the special case when $I$ is generated by a regular sequence, we have
\[ (I^{[q]}:_S I) = (\omega^{q-1}) + I^{[q]},  \]
where $\omega$ denotes the product of minimal generators of $I$. Therefore, to show that $I$ in this case defines an $F$-pure (respectively, $F$-regular ring), it is necessary and sufficient to show that $\omega^{p-1} \notin \mathfrak{m}^{[p]}$ (respectively, that for each $c\in (S/I)^\circ$, we have $c\omega^{q-1} \notin \mathfrak{m}^{[q]}$ for some $q=p^e$). We remark that all the rings that we will prove to be $F$-pure or $F$-regular in this paper are complete intersections.

It turns out that one does not need to check the definition of $F$-regularity of a ring $R$ and in turn, the condition in Glassbrenner's criterion, for every $c\in R^\circ$. It suffices to verify it for just one certain nonzero $c\in R^\circ$, called a \textit{test element}. There are many test elements. In fact, it is easy to see that they form an ideal. One way to obtain a test element is from the singular locus of the ring:

\begin{theorem}\label{thm:poweristestelement}(\cite[Theorem 3.4]{HHFrance89})
    Let $R$ be an $F$-finite reduced ring. Then every element $c\in R^\circ$ such that $R_c$ is regular has a power which is a test element.
\end{theorem}

\begin{theorem}\label{thm:radicaltestideal}
    (\cite[Proposition 2.5]{FW89}) Let $R$ be an $F$-finite and $F$-pure ring. Then the ideal generated by all test elements is radical.
\end{theorem}

Combining these, one obtains:

\begin{corollary}\label{cor:testelement}
    Let $R$ be an $F$-finite and $F$-pure ring. Then every element $c\in R^\circ$ such that $R_c$ is regular is a test element.
\end{corollary}

\section{Proof of Theorem \ref{thm:main}}

Let $\Bbbk$ be a field of prime characteristic $p>0$. To ease the notations, set $A_n\coloneqq \Bbbk[X,Y]/\mathfrak{c}$ and $B_n\coloneqq \Bbbk[X,Y]/\mathfrak{a}$ where $X$ and $Y$ are $n\times n$ matrices of indeterminates, and their commutator matrix  $C$, and $\mathfrak{a}$ and $\mathfrak{c}$ are as defined earlier. To show the $F$-regularity of a ring, it suffices to do so for a faithfully flat extension \cite[Proposition 4.12]{HH90}. Therefore, for the rest of the paper, we will assume $\Bbbk$ is algebraically closed. 

Our strategy is to use deformation techniques and simplify the task at hand by decomposing $\Bbbk[X,Y]/\mathfrak{c}$ into a tensor product of rings of smaller dimensions. Towards this, we will start by proving that some rings of small dimensions are $F$-regular.

\begin{proposition}\label{thm:T}
    Let $W=(w_{ij})$ and $Z=(z_{ij})$ be $3\times 3$ matrices of indeterminates. Then the ring
    \[ 
    T=\Bbbk[W,Z]/(f_1,f_2,f_3,f_4)
    \]
    is $F$-regular, where
    \begin{align*}
        f_1&= w_{21}z_{12} - w_{12}z_{21}+w_{23}z_{32} - w_{32}z_{23} ,\\
        f_2&=w_{22}z_{21} - w_{21}z_{22} +w_{23}z_{31} - w_{31}z_{23} +  w_{21}z_{11} - w_{11}z_{21},\\
        f_3&=w_{13}z_{32} - w_{32}z_{13} - w_{22}z_{12} + w_{12}z_{22} - w_{12}z_{11} + w_{11}z_{12},\\
        f_4&= w_{12}z_{21} - w_{21}z_{12}+ w_{13}z_{31} - w_{31}z_{13} .        
    \end{align*}
\end{proposition}

\begin{proof}
    We first claim that the set of elements
    \begin{equation}\label{equ1}
        \begin{gathered}
            w_{11},\  w_{21},\ w_{23},\ w_{32},\ w_{33}, \ z_{11},\ z_{13},\ z_{22},\ z_{32},\ z_{33},\\
       w_{12}-z_{21},\ w_{13}-z_{31},\ w_{22}-z_{12}, \  w_{31}-z_{23},\\
        f_1,\ f_2,\ f_3,\ f_4
        \end{gathered}
    \end{equation}
    is a homogeneous system of parameters for $\Bbbk[W,Z]$, and thus a regular sequence on $\Bbbk[W,Z]$. Indeed, this set has $18$ elements, and $\dim \Bbbk[W,Z]=18$, hence it suffices to show that the ring $\Bbbk[W,Z]$ modulo the ideal they generate has Krull dimension $0$. This ring is isomorphic to
    \[\frac{\Bbbk[w_{12},\ w_{13},\ w_{22},\ w_{31}]}{(w_{12}^2,\ \ w_{22}w_{12}-w_{31}^2, \ \ w_{22}^2, \ \ w_{13}^2)}, \]
    which is evidently of dimension $0$. In particular, the ring $T$ is a complete intersection.

    Next we show that $T$ is $F$-pure. Since $F$-purity deforms for complete intersections by Theorem \ref{thm:fregdeform}, it suffices to show that the homomorphic image of $T$ modulo the ideal generated by $$w_{11},\ w_{21},\ w_{23},\ w_{32},\ w_{33},\ z_{11},\ z_{13},\ z_{22},\ z_{32},\ z_{33},$$
    a subset of (\ref{equ1}), is $F$-pure. Indeed, this homomorphic image is the complete intersection:
    \[ \frac{\Bbbk[w_{12},\ w_{13},\ w_{22},\ w_{31},\ z_{12},\ z_{21},\ z_{23},\ z_{31}]}{(w_{12}z_{21}, \ \  -w_{22}z_{21}+w_{31}z_{23},\  \ w_{22}z_{12},\ \ w_{13}z_{31})}, \]
    which is evidently $F$-pure by Fedder's criterion since
    \[\big((w_{12}z_{21})(-w_{22}z_{21}+w_{31}z_{23})( w_{22}z_{12})(w_{13}z_{31})\big)^{p-1}  \]
    equals $(w_{12}w_{13}w_{22}w_{31}z_{12}z_{21}z_{23}z_{31})^{p-1}$ modulo $\mathfrak{m}^{[p]}$.

    We turn to $F$-regularity and first need a test element. We compute a maximal minor of the Jacobian matrix of the ideal defining $T$:
    \[
    \det \frac{\partial(f_1,f_2,f_3,f_4)}{\partial(w_{31},w_{32},z_{31},z_{32})}=\det \begin{pmatrix}
        0 & -z_{23} & 0 & -z_{13}\\
        -z_{23} & 0 & -z_{13} & 0\\
        0 & w_{23} & 0 & w_{13}\\
        w_{23} & 0 & w_{13} & 0\\
    \end{pmatrix}=(w_{23}z_{13}-w_{13}z_{23})^2.
    \]
    Hence, by Corollary \ref{cor:testelement}, the polynomial $f\coloneqq w_{23}z_{13}-w_{13}z_{23}$ is a test element once we show that it is in $T^\circ$. Observe that the homomorphic image of $K[W,Z]$ modulo the ideal generated by the $18=\dim K[W,Z]$ elements
		\begin{equation}
		    \begin{gathered}\label{set1}
			w_{11},\  w_{12},\  w_{31},\  w_{33},\  z_{11}, \ z_{21},\  z_{32}, \ z_{33},\\
			w_{13}-z_{31}, \ w_{21}-z_{22},\  w_{23}-z_{13},\  w_{32}- z_{23}, \ w_{13}+ w_{22} + z_{12},\\
			f,\ f_1,\ f_2,\ f_3,\ f_4
		\end{gathered}
		\end{equation}
		is isomorphic to    
		\begin{equation}\label{quotientringforregseq}
		    K[w_{13},\ w_{21}, \ w_{22},\  w_{23},\  w_{32}]/(g,\ g_1, \ g_2,\ g_3,\ g_4)
		\end{equation}
        where 
        \begin{align*}
            g&=w_{23}^2-w_{13}w_{32},\\
            g_1&=w_{32}^2+w_{13}w_{21}+w_{21}w_{22},\\ 
            g_2&=w_{21}^2-w_{13}w_{23},\\ 
            g_3&=w_{22}^2+w_{13}w_{22}-w_{23}w_{32},\\ 
            g_4&=w_{13}^2+w_{13}w_{21}+w_{21}w_{22}.
        \end{align*}
		The radical of the above defining ideal contains $w_{21}$ since
        \begin{multline*}
            w_{21}^4=(
            w_{21}w_{23}+w_{23}^2-w_{13}w_{32}) g_1 + (
            w_{21}^2-w_{22}w_{23}) g_2+ (          
            -w_{23}^2+w_{13}w_{32}+w_{21}w_{32})g_3\\
            + (
            -w_{23}^2+w_{13}w_{32}-w_{22}w_{32} ) g_4  +( 
            w_{13}^2+w_{22}^2-w_{23}w_{32}-w_{32}^2 )g,
        \end{multline*}
        and by examining the defining polynomials, one can show that it also contains $w_{13}, w_{22},  w_{23}$ and $  w_{32}$. The ring (\ref{quotientringforregseq}) is thus of dimension $0$, i.e., (\ref{set1}) forms a homogeneous system of parameters for $\Bbbk[W,Z]$, hence a regular sequence. In particular, $f$ is $T$-regular.
  
    Now we show that $T$ is $F$-regular. By Glassbrenner's criterion, it suffices to show that
    \[ (f)(f_1f_2f_3f_4)^{p-1} \notin \mathfrak{m}^{[p]} \]
    where $\mathfrak{m}$ denotes the homogeneous maximal ideal of $\Bbbk[W,Z]$. This is equivalent to verifying that the image of $(f)(f_1f_2f_3f_4)^{p-1}$ is nonzero in $\Bbbk[W,Z]/\mathfrak{m}^{[p]}$. 
    We will in fact, prove something stronger. For the remainder of the proof, all elements denote their homomorphic images in 
    \begin{equation}\label{homomorphicimageT}
        \Bbbk[W,Z]/\big((w_{11},w_{21},w_{31},w_{33},z_{11},z_{12},z_{32},z_{33})+\mathfrak{m}^{[p]}\big).
    \end{equation}
    We will show that
    $$(w_{23}z_{13})^{p-2} (f)(f_1f_2f_3f_4)^{p-1}  \neq 0.$$
    For degree reasons, this element is a scalar multiple of $(w_{12}w_{13}w_{22}w_{23}w_{32}z_{13}z_{21}z_{22}z_{23}z_{31})^{p-1}$, hence it suffices to find the corresponding coefficient. Note that in the ring (\ref{homomorphicimageT}) one has
    \begin{align*}
        f&=w_{23}z_{13}-w_{13}z_{23},\\
        f_1&=-w_{12}z_{21}-w_{32}z_{23},\\
        f_2&=w_{22}z_{21}+w_{23}z_{31},\\
        f_3&=-w_{32}z_{13}+w_{12}z_{22},\\
        f_4&=w_{12}z_{21}+w_{13}z_{31}.
    \end{align*}
    % write the element more explicitly:
    % \begin{multline*}
    %     (w_{23}z_{13})^{p-2} (w_{23}z_{13}-w_{13}z_{23})(-w_{12}z_{21}-w_{32}z_{23})^{p-1} (w_{22}z_{21}+w_{23}z_{31})^{p-1}\\(-w_{32}z_{13}+w_{12}z_{22})^{p-1} (w_{12}z_{21}+w_{13}z_{31})^{p-1}.
    % \end{multline*}
    Since $w_{22}$ and $z_{22}$ only appear in $f_2$ and $f_3$, respectively, to obtain $w_{22}^{p-1}z_{22}^{p-1}$, the only monomial terms that matter in $f_2^{p-1}$ and $f_3^{p-1}$ are ones divisible by $w_{22}^{p-1}$ and $z_{22}^{p-1}$, respectively. Thus 
    \begin{align*}
		& \ \ \ \ \ (w_{23}z_{13})^{p-2} (f)(f_1f_2f_3f_4)^{p-1} \\
        &=(w_{23}z_{13})^{p-2} (f) (f_1^{p-1}) \big(w_{22}z_{21}\big)^{p-1}\big(w_{12}z_{22}\big)^{p-1} (f_4^{p-1})\\
        &=(w_{12}w_{22}z_{21}z_{22})^{p-1}(w_{23}z_{13})^{p-2} (f) (f_1^{p-1}) (f_4^{p-1}).
	\end{align*}
     Using similar arguments for $f_1^{p-1}$ and $f_4^{p-1}$, the above element is simplified to
    \begin{align*}
        &\ \ \ \ \ (w_{12}w_{22}z_{21}z_{22})^{p-1}(w_{23}z_{13})^{p-2} (f) (-w_{12}z_{21}-w_{32}z_{23})^{p-1} (w_{12}z_{21}+w_{13}z_{31})^{p-1}\\
        &=(w_{12}w_{22}z_{21}z_{22})^{p-1}(w_{23}z_{13})^{p-2} (f) \big( w_{32}z_{23}\big)^{p-1} \big(w_{13}z_{31}\big)^{p-1}\\
        &=(w_{12}w_{13}w_{22} w_{32}z_{21}z_{22}z_{23}z_{31})^{p-1}(w_{23}z_{13})^{p-2} (f). 
    \end{align*}
    Substituting $f$, the above equals
    \begin{align*}
        &\ \ \ \ \ (w_{12}w_{13}w_{22} w_{32}z_{21}z_{22}z_{23}z_{31})^{p-1}(w_{23}z_{13})^{p-2} (w_{23}z_{13}-w_{13}z_{23}) \\
        &=(w_{12}w_{13}w_{22} w_{23}w_{32}z_{13}z_{21}z_{22}z_{23}z_{31})^{p-1},
    \end{align*}
    which is nonzero in the ring (\ref{homomorphicimageT}), as claimed.
\end{proof}

Next we will prove the $F$-regularity of $A_n$ for some small values of $n$. Since the trace of the commutator matrix $C=XY-YX$ is $0$, we can choose the generators of $\mathfrak{c}$ to be the first $n-1$ entries of the diagonal counting from the upper-left corner, together with the anti-diagonal entries of $C$.

\begin{proposition}\label{thm:A3}
    The ring $A_3=\Bbbk[X,Y]/\mathfrak{c}$, where $X$ and $Y$ are $3\times 3$ matrices of indeterminates and $\mathfrak{c}$ denotes the ideal generated by the cross-diagonal of their commutator matrix, is $F$-regular.
\end{proposition}
\begin{proof}
    The ideal $\mathfrak{c}$ is generated by the elements $c_{11},c_{22},c_{31},c_{13},$ where
    \begin{align*}
        c_{11}&=x_{12}y_{21} - x_{21}y_{12} + x_{13}y_{31} - x_{31}y_{13},\\
        c_{22}&=x_{21}y_{12} - x_{12}y_{21} + x_{23}y_{32} - x_{32}y_{23},\\
        c_{31}&=  x_{32}y_{21} -x_{21}y_{32} + x_{31}y_{11} - x_{11}y_{31}  - x_{31}y_{33} + x_{33}y_{31},\\
        c_{13}&=x_{12}y_{23} - x_{23}y_{12} + x_{11}y_{13} - x_{13}y_{11} + x_{13}y_{33} - x_{33}y_{13}.
    \end{align*}
    We observe that \[
    f\coloneqq x_{13}y_{23}(x_{13}y_{31}-x_{31}y_{13}) =- \det \frac{\partial(c_{11},c_{22},c_{31},c_{13})}{\partial(x_{11},x_{32},y_{11},y_{21})}\] is a maximal minor of the Jacobian matrix of the ideal defining $A_3$, and hence by Corollary ~\ref{cor:testelement} and the fact that $A_3$ is $F$-pure (Theorem \ref{Fsingularities}. (4)), $f$ is a test element once we show that it is in $A_3^\circ$. Indeed, the homomorphic image of $\Bbbk[X,Y]$ modulo the ideal generated by the $18=\dim \Bbbk[X,Y]$ elements
    \begin{equation}\label{set2}
        \begin{gathered}
    	x_{12}, \ x_{22}, \ x_{31},\  x_{33},\  y_{11}, \ y_{22}, \ y_{33},\\
    	x_{11}-y_{13},\  x_{13}-y_{23},\  x_{13}-y_{31}, \ x_{21}- y_{12}, \ x_{23}-y_{32},\  x_{32}-y_{21},\\
    	c_{11}, \ c_{22}, \ c_{31},\  c_{13},\  f
    \end{gathered}
    \end{equation}
    is isomorphic to
    \[ 
    \frac{\Bbbk[x_{11},\ x_{13},\ x_{21},\ x_{23},\ x_{32}]}{(x_{13}^2-x_{21}^2, \ \ x_{21}^2+ x_{23}^2-x_{13}x_{32},\ \  -x_{11}x_{13}- x_{21}x_{23}+ x_{32}^2,\ \  x_{11}^2- x_{21}x_{23},\ \   x_{13}^4
    	)}
    \]
    which is evidently of dimension $0$, i.e., (\ref{set2}) forms a homogeneous system of parameters for $\Bbbk[X,Y]$, hence a regular sequence. In particular, $f$ is $A_3$-regular. 
    
    By Glassbrenner's criterion, to prove the $F$-regularity of $A_3$, it suffices to show that
    \begin{equation}\label{glassbrennerA3}
        (f)(c_{11}c_{22}c_{31}c_{13})^{p^2-1} \notin \mathfrak{m}^{[p^2]}
    \end{equation}
    where $\mathfrak{m}$ denotes the homongeneous maximal ideal of $\Bbbk[X,Y]$.
    We will prove something stronger. For the remainder of the proof, all elements denote their homomorphic images in \begin{equation}\label{homomorphicimageA3}
        \Bbbk[X,Y]/\big( (x_{12},x_{22},x_{31},x_{33}, y_{11},y_{22},y_{33})+\mathfrak{m}^{[p^2]}\big).
    \end{equation}
    We will show that the element
    $$(x_{13}^{p^2-3})(y_{23}y_{31})^{p^2-2}(f)(c_{11}c_{22}c_{31}c_{13})^{p^2-1}$$
    is nonzero in the ring (\ref{homomorphicimageA3}).
    For degree reasons, this element equals a scalar multiple of $(x_{11}x_{13}x_{21}x_{23}x_{32}y_{12}y_{13}y_{21}y_{23}y_{31}y_{32})^{p^2-1}$. In the ring (\ref{homomorphicimageA3}) one has
    \begin{align*}
        f&=x_{13}^2y_{23}y_{31},\\
        c_{11}&=-x_{21}y_{12}+x_{13}y_{31},\\
        c_{22}&=x_{21}y_{12}+x_{23}y_{32}-x_{32}y_{23},\\
        c_{31}&=x_{32}y_{21}-x_{21}y_{32}-x_{11}y_{31},\\
        c_{13}&=-x_{23}y_{12}+x_{11}y_{13}.
    \end{align*}
    We first substitute $f$:
    \begin{align*}
    	&\ \ \ \ \ (x_{13}^{p^2-3})(y_{23}y_{31})^{p^2-2}(f)(c_{11}c_{22}c_{31}c_{13})^{p^2-1} \\
    	&=(x_{13}y_{23}y_{31})^{p^2-1}(c_{11}c_{22}c_{31}c_{13})^{p^2-1}
    \end{align*}
    Since we already have the factor $(x_{13}y_{23}y_{31})^{p^2-1}$, any additional monomial term that contains one of $x_{13}, y_{23},$ and $y_{31}$ may be ignored. Substituting $c_{11}, c_{22}, c_{31},$ and $c_{13}$ with this observation, one has
    \begin{align*}
    	&\ \ \ \ \  \!\begin{multlined}[t][.95\displaywidth]
    		(x_{13}y_{23}y_{31})^{p^2-1}\big( -x_{21}y_{12}\big)^{p^2-1}
    			(x_{21}y_{12}+x_{23}y_{32})^{p^2-1}(x_{32}y_{21}-x_{21}y_{32})^{p^2-1}
    			\\(-x_{23}y_{12}+x_{11}y_{13})^{p^2-1}
    	\end{multlined}\\
    	&= \!\begin{multlined}[t][.95\displaywidth]
    		(x_{13}x_{21} y_{12}y_{23}y_{31})^{p^2-1}
    			(x_{21}y_{12}+x_{23}y_{32})^{p^2-1}(x_{32}y_{21}-x_{21}y_{32})^{p^2-1}
    			(-x_{23}y_{12}+x_{11}y_{13})^{p^2-1}
    	\end{multlined}.
    \end{align*}
    We have a new monomial factor $(x_{21}y_{12})^{p^2-1}$. Using similar arguments, we obtain
    \begin{align*}
    	&\ \ \ \ \  \!\begin{multlined}[t][.95\displaywidth]
    		(x_{13}x_{21} y_{12}y_{23}y_{31})^{p^2-1} \big(x_{23}y_{32}\big)^{p^2-1}\big(x_{32}y_{21}\big)^{p^2-1}
    			\big(x_{11}y_{13}\big)^{p^2-1}
    	\end{multlined}\\
    	&=(x_{11}x_{13}x_{21}x_{23}x_{32}y_{12}y_{13}y_{21}y_{23}y_{31}y_{32})^{p^2-1}
    \end{align*}
    which is nonzero in the ring (\ref{homomorphicimageA3}), as claimed.
\end{proof}

\begin{proposition}\label{thm:A4}
    The ring $A_4=\Bbbk[X,Y]/\mathfrak{c}$, where $X$ and $Y$ are $4\times 4$ matrices of indeterminates and $\mathfrak{c}$ denotes the ideal generated by the cross-diagonal of their commutator matrix, is $F$-regular.
\end{proposition}
\begin{proof}
	The ideal $\mathfrak{c}$ is generated by $\{c_{11},c_{22},c_{33},c_{41},c_{32},c_{23},c_{14}\}$ where
	\begin{align*}
		c_{11}&=x_{12}y_{21} - x_{21}y_{12} + x_{13}y_{31} - x_{31}y_{13} + x_{14}y_{41} - x_{41}y_{14},\\
		c_{22}&=x_{21}y_{12} - x_{12}y_{21} + x_{23}y_{32} - x_{32}y_{23} + x_{24}y_{42} - x_{42}y_{24},\\
		c_{33}&=x_{31}y_{13} - x_{13}y_{31} - x_{23}y_{32} + x_{32}y_{23} + x_{34}y_{43} - x_{43}y_{34},\\
		c_{41}&=x_{41}y_{11} - x_{11}y_{41} - x_{21}y_{42} + x_{42}y_{21} - x_{31}y_{43} + x_{43}y_{31} - x_{41}y_{44} + x_{44}y_{41},\\
		c_{32}&=x_{31}y_{12} - x_{12}y_{31} - x_{22}y_{32} + x_{32}y_{22} - x_{32}y_{33} + x_{33}y_{32} + x_{34}y_{42} - x_{42}y_{34},\\
		c_{23}&=x_{21}y_{13} - x_{13}y_{21} + x_{22}y_{23} - x_{23}y_{22} + x_{23}y_{33} - x_{33}y_{23} + x_{24}y_{43} - x_{43}y_{24},\\
		c_{14}&=x_{11}y_{14} - x_{14}y_{11} + x_{12}y_{24} - x_{24}y_{12} + x_{13}y_{34} - x_{34}y_{13} + x_{14}y_{44} - x_{44}y_{14}.
	\end{align*}
	We observe that \[
 f\coloneqq x_{12}x_{13}y_{24}(x_{14}y_{41}-x_{41}y_{14})(x_{23}y_{32}-x_{32}y_{23})=- \det \frac{\partial(c_{11},c_{22},c_{33},c_{41},c_{32},c_{23},c_{14})}{\partial(x_{11},x_{22},x_{42},y_{11},y_{21},y_{22},y_{31})}\] is a maximal minor of the Jacobian matrix of the ideal defining $A_4$, and thus by Corollary ~\ref{cor:testelement} and the fact that $A_4$ is $F$-pure (Theorem \ref{Fsingularities}. (4)), $f$ is a test element once we show that it is in $A_4^\circ$. Observe that 
	the homomorphic image of $\Bbbk[X,Y]$ modulo the ideal generated by the $32=\dim \Bbbk[X,Y]$ elements
    \begin{equation}\label{set3}
         \begin{gathered}
		x_{31},\ x_{32},\  x_{33},\  x_{41},\  x_{44},\  y_{11},\  y_{13},\  y_{21},\  y_{22}, \ y_{33},\  y_{44},\\
		x_{12}-x_{13},\  x_{12}-x_{14}, \ x_{12}-x_{23},\\
		x_{11}-y_{14}, \ x_{12}-y_{24}, \ x_{12}-y_{32},\  x_{12}-y_{41},\\
		x_{21}-y_{12}, \ x_{22}-y_{23},\  x_{24}-y_{42},\\
		 x_{34}-y_{43},\  x_{42}-y_{34}, \ x_{43}-y_{31},\\
		 c_{11},\ c_{22},\ c_{33},\ c_{41},\ c_{32},\ c_{23},\ c_{14},\  f
	\end{gathered}
    \end{equation}
 is isomorphic to
	\[
	\frac{\Bbbk[x_{11},\ x_{12},\ x_{21},\ x_{22},\ x_{24},\ x_{34},\ x_{42},\ x_{43}]}{\begin{multlined}
			(x_{12}^2-x_{21}^2+x_{12}x_{43}, \ \ x_{12}^2+x_{21}^2+x_{24}^2-x_{12}x_{42},
			\ \ -x_{12}^2+x_{34}^2-x_{12}x_{43}-x_{42}x_{43},\\
			 -x_{11}x_{12}-x_{21}x_{24}+x_{43}^2,
			\ \ -x_{12}x_{22}+x_{24}x_{34}-x_{42}^2-x_{12}x_{43},\\
			x_{22}^2+x_{24}x_{34}-x_{12}x_{43},
			\ \ x_{11}^2+x_{12}^2-x_{21}x_{24}+x_{12}x_{42},
			\ \ x_{12}^7)
	\end{multlined}} 
	\]
	which is evidently of dimension $0$, i.e., the elements (\ref{set3}) form 
	a homogeneous system of parameters for $\Bbbk[X,Y]$, hence a regular sequence. In particular, $f$ is $A_4$-regular.  By Glassbrenner's criterion, to prove the $F$-regularity of $A_4$, it suffices to show that 
	\[ (f)(c_{11}c_{22}c_{33}c_{41}c_{32}c_{23}c_{14})^{p-1} \notin \mathfrak{m}^{[p]} \]
	where $\mathfrak{m}$ denotes the homogeneous maximal ideal of $\Bbbk[X,Y]$. We will prove something stronger, that the element
 \begin{equation}\label{elementA4}
     (x_{12}x_{13}x_{32}x_{41}y_{14}y_{23}y_{24})^{p-2}(f)(c_{11}c_{22}c_{33}c_{41}c_{32}c_{23}c_{14})^{p-1}
 \end{equation}
	where all elements, for the remainder of the proof, denote their homomorphic images in the ring
 \begin{equation}\label{homomorphicA4}
     \Bbbk[X,Y]/\big( (x_{11},x_{22},x_{24},x_{33},x_{44}, y_{31},y_{32},y_{33}, y_{41}, y_{43},y_{44})+\mathfrak{m}^{[p]}\big).
 \end{equation}
    We replace the polynomials with their corresponding homomorphic images:
    \begin{align*}
        f&=x_{12}x_{13}x_{32}x_{41}y_{14}y_{23}y_{24}\\
        c_{11}&=x_{12}y_{21} - x_{21}y_{12}  - x_{31}y_{13} - x_{41}y_{14},\\
		c_{22}&=x_{21}y_{12} - x_{12}y_{21} - x_{32}y_{23} - x_{42}y_{24},\\
		c_{33}&=x_{31}y_{13} + x_{32}y_{23}  - x_{43}y_{34},\\
		c_{41}&=x_{41}y_{11} - x_{21}y_{42} + x_{42}y_{21},\\
		c_{32}&=x_{31}y_{12} + x_{32}y_{22}  + x_{34}y_{42} - x_{42}y_{34},\\
		c_{23}&=x_{21}y_{13} - x_{13}y_{21}  - x_{23}y_{22} - x_{43}y_{24},\\
		c_{14}&=-x_{14}y_{11} + x_{12}y_{24} + x_{13}y_{34} - x_{34}y_{13}.
    \end{align*}
    For degree reasons, the element $(\ref{elementA4})$ equals  a scalar multiple of  $$(x_{12}x_{13}x_{14}x_{21}x_{23}x_{31}x_{32}x_{34}x_{41}x_{42}x_{43}y_{11}y_{12}y_{13}y_{14}y_{21}y_{22}y_{23}y_{24}y_{34}y_{42})^{p-1}.$$
    We first substitute $f$ in the element $(\ref{elementA4})$:
	\begin{align*}
		&\ \ \ \ \  (x_{12}x_{13}x_{32}x_{41}y_{14}y_{23}y_{24})^{p-2}(f)(c_{11}c_{22}c_{33}c_{41}c_{32}c_{23}c_{14})^{p-1}\\
		&=(x_{12}x_{13}x_{32}x_{41}y_{14}y_{23}y_{24})^{p-1}(c_{11}c_{22}c_{33}c_{41}c_{32}c_{23}c_{14})^{p-1}.
        \end{align*}
    Since we already have the factor $(x_{12}x_{13}x_{32}x_{41}y_{14}y_{23}y_{24})^{p-1}$, any additional monomial term that contains one of these variables may be ignored. Substituting $c_{22}$ with this observation, we obtain 
        \begin{align*}
		&\ \ \ \ \ (x_{12}x_{13}x_{32}x_{41}y_{14}y_{23}y_{24})^{p-1}\big(c_{11}(x_{21}y_{12})c_{33}c_{41}c_{32}c_{23}c_{14}\big)^{p-1} \\
		&=(x_{12}x_{13}x_{21}x_{32}x_{41}y_{12}y_{14}y_{23}y_{24})^{p-1}(c_{11}c_{33}c_{41}c_{32}c_{23}c_{14})^{p-1}
        \end{align*}
    We have a new monomial factor $(x_{21}y_{12})^{p-1}$. Using similar arguments for $c_{11}, c_{41},$ and $c_{23}$, the above element equals
        \begin{align*}
		&\ \ \ \ \ \!\begin{multlined}[t][.95\displaywidth]
			(x_{12}x_{13}x_{21}x_{32}x_{41}y_{12}y_{14}y_{23}y_{24})^{p-1}\big((-x_{31}y_{13})c_{33}(x_{42}y_{21})c_{32}(-x_{23}y_{22})c_{14}\big)^{p-1}
		\end{multlined}\\
		&=\!\begin{multlined}[t][.95\displaywidth]
			(x_{12}x_{13}x_{21}x_{23}x_{31}x_{32}x_{41}x_{42}y_{12}y_{13}y_{14}y_{21}y_{22}y_{23}y_{24})^{p-1}\big(c_{33}c_{32}c_{14}\big)^{p-1}
		\end{multlined}
        \end{align*}
    With the new monomial factor $(x_{23}x_{31}x_{42}y_{13}y_{21}y_{22})^{p-1}$, we simplify $c_{33},c_{32},$ and $ c_{14}$:
        \begin{align*}
		&\ \ \ \ \ \!\begin{multlined}[t][.95\displaywidth]
			(x_{12}x_{13}x_{21}x_{23}x_{31}x_{32}x_{41}x_{42}y_{12}y_{13}y_{14}y_{21}y_{22}y_{23}y_{24})^{p-1}\big((-x_{43}y_{34})(x_{34}y_{42})(-x_{14}y_{11})\big)^{p-1}
		\end{multlined}\\
		&=(x_{12}x_{13}x_{14}x_{21}x_{23}x_{31}x_{32}x_{34}x_{41}x_{42}x_{43}y_{11}y_{12}y_{13}y_{14}y_{21}y_{22}y_{23}y_{24}y_{34}y_{42})^{p-1},
	\end{align*}
        which is nonzero in the ring $(\ref{homomorphicA4})$, as claimed.
\end{proof}

\begin{proof}[Proof of Theorem \ref{thm:main}]
    Recall that $A_n=\Bbbk[X,Y]/\mathfrak{c}$, where $X$ and $Y$ are $n\times n$ matrices of indeterminates and $\mathfrak{c}$ denotes the ideal generated by the cross-diagonal of their commutator matrix. The rings $A_1$ and $A_2$ are $F$-regular \cite[pages 201-202]{Ka18}. We now assume $n\geq 3$. It is known that $A_n$ is an $F$-pure complete intersection \cite[Theorem 3.4, Theorem 3.7, Theorem 3.9]{KY22}. In particular, we have
    \[\dim\ A_{2k-1} = 8k^2-12k+6\]
    and 
    \[\dim\ A_{2k}= 8k^2-4k+1\]
    where $k\geq 2$.   We will show that the rings $A_n$ are $F$-regular by induction on $n$.  We already proved the base cases that $A_3$ and $A_4$ are $F$-regular in Propositions \ref{thm:A3} and \ref{thm:A4}. Assume that $A_m$ is $F$-regular for any $m<n$. We will prove that $A_n$ is $F$-regular. If $n=2k+1$ for some $k\geq 2$, set
    \[
    \Omega_1\coloneqq \{x_{ij},y_{ij} \mid (i,j)=(k+1,l) \text{ or }(l,k+1) \text{ where } 1\leq l < n \}
    \]
    We have
    \begin{multline*}
    	\frac{A_{n}}{(\Omega_1)}\cong \frac{\Bbbk[X',Y']}{\mathfrak{c}(X'Y'-Y'X')} \otimes_K \frac{K\begin{bmatrix}
    			x_{n,k+1} & x_{k+1,n}\\
    			y_{n,k+1} & y_{k+1,n}
    	\end{bmatrix}}{(x_{n,k+1}y_{k+1,n}- x_{k+1,n}y_{n,k+1} ) }
    	\cong\ \  A_{n-1} \otimes_K \frac{\Bbbk[w_1,\ w_2, \ w_3, \ w_4]}{(w_1w_4-w_2w_3) },
    \end{multline*}
	where $X'$ and $Y'$ are obtained from $X$ and $Y$, respectively, by deleting the $(k+1)^{\text{th}}$ row and column, and $w_1,w_2,w_3,w_4$ are new indeterminates. Since determinantal rings are $F$-regular \cite[Theorem 7.14]{HH94b}, the ring on the right-hand side is $F$-regular due to \cite[Theorem 7.45]{HH94a} and the induction hypothesis, and thus so is $A_{n}/(\Omega_1)$. To show that $A_{n}$ is $F$-regular, it suffices to show that $\Omega_1$ forms a regular sequence on $A_{n}$ by Theorem \ref{thm:fregdeform}. Since $A_n$ is a complete intersection, it suffices to show that the dimension of $A_n$ drops by exactly $|\Omega_1|$ after taking modulo by the ideal generated by $\Omega_1$. Indeed, we have
    \begin{align*}
        \dim\ A_n/(\Omega_1)+ |\Omega_1| &= \dim\ A_{n-1} + \dim \frac{\Bbbk[w_1,\ w_2, \ w_3, \ w_4]}{(w_1w_4-w_2w_3) } + 2(4k-1) \\
        &= \dim\ A_{2k} + 3 + 2(4k-1)\\
        &= (8k^2-4k+1) +3 + 2(4k-1) \\
        &=\dim \ A_{2k+1}\\
        &=\dim\ A_n.
    \end{align*}
    Now suppose $n=2k+2$ for some $k\geq 2$. Set
	\[\Omega_2\coloneqq \{x_{ij},y_{ij}\mid (i,j)=(k,l), (k+1,l), (l,k) \text{ or } (l,k+1) \text{ where } 1\leq l<k \text{ or } k+1<l< n\}. \]
	We have
	\begin{equation*}
		\frac{A_{n}}{(\Omega_2)} \cong \!\begin{multlined}[t][.9\displaywidth]
			\frac{\Bbbk[X'',Y'']}{\mathfrak{c}(X''Y''-Y''X'')}\otimes_K 
			T'
		\end{multlined}		\cong
\ A_{n-2}\otimes_K T'
	\end{equation*}
	where $X''$ and $Y''$ are obtained from $X$ and $Y$, respectively, by deleting the $k^{\text{th}}$ and $(k+1)^{\text{th}}$ rows and columns and 
	\begin{multline*}
		T'\coloneqq \frac{K\begin{bmatrix}
				x_{k,k}& x_{k,k+1}& x_{k,n}\\
				x_{k+1,k}& x_{k+1,k+1}& x_{k+1,n}\\
				x_{n,k}& x_{n,k+1}& 0\\
			\end{bmatrix} \begin{bmatrix}
				y_{k,k}& y_{k,k+1}& y_{k,n}\\
				y_{k+1,k}& y_{k+1,k+1}& y_{k+1,n}\\
				y_{n,k}& y_{n,k+1}& 0\\
		\end{bmatrix}}{\begin{multlined}
				\big(x_{k+1,n}y_{n,k+1} - x_{n,k+1}y_{k+1,n} + x_{k+1,k}y_{k,k+1} - x_{k,k+1}y_{k+1,k},\\
				x_{k+1,n}y_{n,k} - x_{n,k}y_{k+1,n} + x_{k+1,k+1}y_{k+1,k} - x_{k+1,k}y_{k+1,k+1} + x_{k+1,k}y_{k,k} - x_{k,k}y_{k+1,k},\\
				x_{k,n}y_{n,k+1} - x_{n,k+1}y_{k,n} - x_{k+1,k+1)}y_{k,k+1} + x_{k,k+1}y_{k+1,k+1} - x_{k,k+1}y_{k,k} + x_{k,k}y_{k,k+1},\\
				x_{k,n}y_{n,k} - x_{n,k}y_{k,n} - x_{k+1,k}y_{k,k+1} + x_{k,k+1}y_{k,n}\big) 
		\end{multlined}}.
	\end{multline*}
	Since $T'[w_{33},z_{33}]=T$ where $T$ denotes the ring in Proposition \ref{thm:T}, the ring $T'$ is a direct summand of $T$, an $F$-regular ring. Thus $T'$ is $F$-regular by \cite[Proposition 4.12]{HH90} and hence so is $A_n/(\Omega_2)$ due to \cite[Theorem 7.45]{HH94a} and the induction hypothesis. By arguments similar to the previous case, one can show that the elements in $\Omega_2$ form a regular sequence on $A_n$, and thus $A_n$ is $F$-regular by Theorem \ref{thm:fregdeform}.
	
	Regarding $B_n=\Bbbk[X,Y]/\mathfrak{a}$, where $X$ and $Y$ are $n\times n$ matrices of indeterminates, it is known that $B_1$ is $F$-regular and  $B_2$ is not even a domain, let alone $F$-regular \cite[page ~202]{Ka18}. For $n\geq 3$, since $A_n$ is a complete intersection \cite[Theorem 3.4]{KY22}, it is isomorphic to $B_n$ modulo a regular sequence. Since $A_n$ is $F$-regular, so is $B_n$ for any $n\geq 3$ by Theorem ~\ref{thm:fregdeform}.
\end{proof}

\section*{Acknowledgement} The author was supported by the NSF grants DMS 1801285, 2101671, and 2001368. The author would like to thank his advisor Anurag Singh for suggesting this problem and for the constant encouragement and helpful discussions. The author would like to thank his advisor Srikanth Iyengar for carefully reading the many versions of this paper.
\bibliographystyle{amsplain}
\bibliography{refs}
\end{document}